\documentclass[letterpaper, 10 pt, conference]{ieeeconf}
\IEEEoverridecommandlockouts                              
\usepackage{amsthm}
\usepackage{enumerate}
\usepackage{amsfonts} 
\usepackage{amsmath} 
\usepackage{algorithm}
\newcommand{\N}{\mathcal{N}}
\usepackage{algorithmic}
\usepackage{amssymb}
\usepackage{xcolor}
\usepackage{url}
\usepackage{lipsum,adjustbox}
\usepackage{tikz}
\usepackage{multirow}
\let\labelindent\relax
\usepackage[shortlabels]{enumitem}
\renewcommand{\theenumi}{(\roman{enumi})}

\usepackage{cite}
\newtheorem{theorem}{Theorem}[section]
\newtheorem{proposition}[theorem]{Proposition}

\theoremstyle{remark}
\newtheorem{remark}{Remark}
\theoremstyle{definition}

\DeclareMathOperator{\Lc}{\mathcal{L}}

\DeclareMathOperator{\real}{\mathbb{R}}

\newcommand{\longthmtitle}[1]{\mbox{} \textit{(#1):}}

\renewcommand{\baselinestretch}{1.025}

\title{Accelerated Algorithms for a Class of Optimization Problems
with Equality and Box Constraints
\thanks{The first author was supported by a Siemens Fellowship. This work was supported in part by Boeing Strategic University Initiative, and in part by the NSF CPS-DFG Joint award 1932406.}}

\author{Anjali Parashar \quad  Priyank Srivastava \quad Anuradha M. Annaswamy 
\thanks{A. Parashar, P. Srivastava and A. M. Annaswamy are with the Department of Mechanical Engineering, Massachusetts Institue of Technology.
Email:\texttt{\{anjalip,psrivast,aanna\}@mit.edu}.
}
}

\begin{document}
\maketitle

\begin{abstract}
Convex optimization with equality and inequality constraints is a ubiquitous problem in several optimization and control problems in
large-scale systems. Recently there has been a lot of interest in
establishing accelerated convergence of the loss function. A class of
high-order tuners was recently proposed in an effort to lead to accelerated convergence for the case when no constraints
are present. In this paper, we propose a new high-order tuner that can
accommodate the presence of equality constraints. In order to
accommodate the underlying box constraints, time-varying gains are
introduced in the high-order tuner which leverage convexity and ensure
anytime feasibility of the constraints. Numerical examples are provided to
support the theoretical derivations.
\end{abstract}
\thispagestyle{empty}
\section{Introduction}
A class of algorithms referred to as High Order Tuners (HT) was proposed for parameter estimation involved in dynamic systems~\cite{Gaudio2020ASystems}. These iterative algorithms utilize second-order information of the system, enabling faster convergence in discrete time parameter estimation. This property of HT is motivated by Nesterov's algorithm~\cite{Nesterov2018LecturesOptimization},~\cite{JMLR:v17:15-084}. 
Additionally, HT has been proven to show stable performance in the presence of time-varying regressors~\cite{Gaudio2020ASystems} and for adaptive control in the presence of delays and high relative degrees~\cite{evesque2003adaptive},\cite{morse1992high}. To further leverage the accelerated convergence of HT in machine learning as well as adaptive control problems, discrete-time version of these algorithms for unconstrained convex optimization was proposed in~\cite{Jose2021-LCSS}, which demonstrated theoretical guarantees pertaining to convergence and stability.

Accelerated algorithms have been extended to equality-constrained convex optimization problems in~\cite{PS-JC:21-csl}, \cite{7963773}, \cite{xu2017accelerated}.  Accelerated  HT was extended to a class of constrained convex optimization problems, with a dual objective of deconstructing the loss landscapes of constrained optimization problems and applying HT with theoretical guarantees of stable performance in~\cite{9993120}. Several works in the past~\cite{PLD-DR-JZK:21} have utilized the variable reduction approach to transform a constrained convex optimization problem to an unconstrained optimization problem in a reduced dimension feasible solution space. However, this transformation does not guarantee convexity of the unconstrained optimization problem. In this paper, we present conditions on equality constraints under which the transformed problem is convex.

We first focus on optimization problems with equality constraints and elucidate the conditions under which HT can be deployed with theoretical guarantees pertaining to convergence and stability, using a variable reduction based technique. We then extend these results to optimization with equality constraints and box inequality constraints. Since the guarantees of convexity are only valid within the feasible region, it makes the task of solving hard inequality constraints challenging. To preserve theoretical guarantees pertaining to convergence, it is imperative to constrain the decision variable within a compact set within which the resulting loss function is convex. We show that the same HT proposed in~\cite{Gaudio2020ASystems},~\cite{Jose2021-LCSS} and ~\cite{9993120} can be used to guarantee feasibility and convergence to the optimal solution even in the presence of equality and box constraints. We present two numerical example problems in the paper to validate the approach.

The organization of the paper is as follows. Section II describes the broader category of constrained optimization problems and presents theorems pertaining to the performance of High Order Tuner (Algorithm 1) in solving them. In Section III, we present a specific category of constrained optimization problems where the loss function is not convex everywhere, but over known compact sets. We outline a novel formulation based on High Order Tuner with time-varying gains (Algorithm 2) to solve such problems. In Section IV, we look at a scalar example that validates the implementation of Algorithm 2 to a constrained convex optimization problem and demonstrates the accelerated convergence of High Order Tuner. The paper concludes with a summary of main contributions and future directions of research in Section V.  
\subsection*{Notation} We employ the following notations throughout the paper. 
For a vector $x \in \real^n$, with $1\le i<j \le n$, $x_{i:j}$ denotes the subvector with elements from the $i$-th entry of $x$ to the $j$-th entry.
For a vector-valued function $p:\real^m \to \real^n$, $p$ is convex (respectively, concave) on $\real^m$ implies that scalar function $p_i:\real^m \to \real$ is convex (respectively, concave) on $\real^m$ for all $i \in \{1,\ldots, n\}$.
We use the shorthand notation PSD to denote a symmetric positive semidefinite matrix.
For two sets $A$ and $B$, $A \times B$ denotes their cartesian product.

\section{Convex Optimization for a Class of Nonconvex Problems}\label{sec:nonlinear}
Consider the optimization problem
\begin{equation}\label{eq:problem-equality-nl}
    \begin{aligned}
        \min \ &f(x) \\
        \text{s.t.} \enspace & h(x) = 0,
    \end{aligned}
\end{equation}
here $x \in \real^n$ is the decision variable,
$f:\real^n \to \real$, $h : \real^n \to \real^{n-m}$ are continuously differentiable convex functions.
Without loss of generality, we assume that problem~\eqref{eq:problem-equality-nl} is not overdetermined, i.e.,  $m \le n$.
There are two key aspects to solving a constrained optimization problem given by~\eqref{eq:problem-equality-nl}; namely, feasibility and optimality. To ensure that the solution is feasible, we first address how to solve equality constraints using variable reduction technique employed in~\cite{9993120}.

We define $\Lc(\cdot)$ as the loss function consisting of the objective function and an optional term penalizing the violation of equality constraints in case it is not feasible to solve the equality constraints:
\begin{align}\label{eq:loss-equality-nl}
    \Lc(x)=f(x)+\lambda_h \|h(x)\|^2, \lambda_h \geq 0
\end{align}
In this paper, we restrict our attention to problems where $\Lc$ is convex, which holds for many cases where $f$ and $h$ are convex. 
To ensure feasibility, we utilize the linear dependence in the feasible solution space introduced by the equality constraints, as illustrated in~\cite{PLD-DR-JZK:21}. $x$ is partitioned into independent and dependent variables; $\theta \in \real^m$ and $z \in \real^{n-m}$ respectively
\begin{equation*}
    x = [\theta^T \quad z^T]^T, \quad z = p(\theta)
\end{equation*}  
Here $p:\real^m \to \real^{n-m}$ is a function that maps the dependence of $z$ on $\theta$, such that $h(x) = 0$. We assume that $h$ is such that given $m$ entries of $x$, its remaining $(n-m)$ entries can be computed either in closed form or recursively. If $p(\theta)$ can be computed explicitly, we set $\lambda_h =0$ in~\eqref{eq:loss-equality-nl}. Otherwise, $\lambda_h$ is chosen as a positive real-valued scalar. In other words, we assume that we have knowledge of the function \(p:\mathbb{R}^m \rightarrow \mathbb{R}^{n-m}\)  such that 
\begin{equation*}
    h([\theta^T \quad p(\theta)^T]^T) = 0 , \quad \forall \theta \in \mathbb{R}^m
\end{equation*} 
Using the function $p(\cdot)$ defined as above, we now define a modified loss function $l:\mathbb{R}^m \rightarrow \mathbb{R}$ as
\begin{equation}\label{eq:l_theta_expansion}
    l(\theta) = \mathcal{L}([\theta^T \quad p(\theta)^T]^T).
\end{equation}

The optimization problem  in~\eqref{eq:problem-equality-nl} is now reformulated as
an unconstrained minimization problem given by
\begin{equation}\label{eq:unconstrained-theta}
    \min \: l(\theta),
\end{equation}
with $\theta \in \mathbb{R}^m$ as the decision variable.
We now proceed to delineate conditions under which $l$ is convex in Proposition~\ref{prop:cvx-nl}.
\begin{proposition}\label{prop:cvx-nl}
\longthmtitle{Convexity of the modified loss function for equality-constrained nonconvex programs}
Assume that there exists a convex set $\Omega_n \in \real^n$ such that
 the functions $f$ and $h$ are convex on $\Omega_n$.
Let 
\begin{align}\label{eq:omega-m}
    \Omega_m=\{\theta \: | \: \theta=x_{1:m}, x \in \Omega_n \}.
\end{align}
If $\Lc$ is convex and any of the following conditions is satisfied:
\begin{enumerate}
    \item  $h$ is linear,
    \item $\nabla \Lc(x) \ge 0$ for all $x \in \Omega_n$ and $p$ is convex on $\Omega_m$,
    \item $\nabla \Lc(x) \le 0$ for all $x \in \Omega_n$ and $p$ is concave on $\Omega_m$,
    \end{enumerate}
    then $l$ is convex on $\Omega_m$.
\end{proposition}
\begin{proof}
Readers are referred to~\cite[Proposition IV.2]{9993120} for the proof and~\cite{boyd_vandenberghe_2004} for details on convexity of composite functions used in the proof.
\end{proof}
Using Implicit Function theorem, we can establish the following sufficient conditions on $h(\cdot)$ for which $p(\cdot)$ is convex or concave. This gives us a readily checkable set of conditions of convexity of $l(\cdot)$ without needing to determine the convexity/concavity of function $p(\cdot)$, which may not always be possible.

\begin{proposition}\label{prop:equality_constraints}
\longthmtitle{Conditions on h(x) for convexity of $p(\theta)$}
Following from Proposition~\ref{prop:cvx-nl}, assuming f is convex and h is \textit{strictly} convex on a given set $\Omega_n$, and $h_i$ is twice differentiable $\forall \ i \in \mathbb{N}$, it follows that:
\begin{enumerate}
    \item if $\nabla_p h(x) < 0$ then $p(\theta)$ is \textit{strictly} convex on $\Omega_m$
    \item if $\nabla_p h(x) > 0$ then $p(\theta)$ is \textit{strictly} concave on $\Omega_m$
\end{enumerate}
Additionally, if h is linear in $z$ and strictly convex in $\theta$, then condition (i) follows.
\\
(refer to the Remark for $p(\theta)$)
\end{proposition}
\begin{proof}
We prove condition (i), similar arguments can be extended to prove condition (ii). 
Noting that 
\begin{equation}\label{eq:h_eq}
    h([\theta^T \quad p(\theta)^T]^T) = 0 ,
\end{equation} 
and applying the chain rule and differentiating~\eqref{eq:h_eq} along the manifold $z=p(\theta)$ twice, for $1 \le i \le m$ and $1 \le j \le n-m$ we get:
\begin{equation}\label{eq:h_expand}
    \underbrace{\frac{\partial^2 h_i}{\partial \theta^2}}_{I} + \underbrace{\frac{\partial h_i}{\partial z_j}\frac{\partial^2 z_j}{\partial \theta^2}}_{II} + \underbrace{\frac{\partial^2 h_i}{\partial z_j^2}\frac{\partial z_j}{\partial \theta}\biggl(\frac{\partial z_j}{\partial \theta}}_{III}\biggr)^T = 0
\end{equation}
Since $h_i$ is convex on $\Omega_n \ \forall \ i \in \mathbb{N}$, it follows: 
\begin{enumerate}
    \item I: $\frac{\partial^2 h_i}{\partial \theta^2}$ is a positive semi-definite matrix
    \item III: $\frac{\partial z_j}{\partial \theta}\biggl(\frac{\partial z_j}{\partial \theta}\biggr)^T$ is a symmetrical dyad, i.e., PSD matrix multiplied by $\nabla^2_{z_j} h_i \geq 0$, hence III is a PSD matrix.
\end{enumerate}
For~\eqref{eq:h_expand} to be valid, II has to be a negative semi-definite matrix, since it is the negated sum of positive-semi definite matrices. From condition (i) we have $\frac{\partial h_i}{\partial z_j} < 0$ since $z = p(\theta)$. Thus, $\frac{\partial^2 z_j}{\partial \theta^2}$ is PSD. Therefore, $\frac{\partial^2 p_j(\theta)}{\partial \theta^2}$ is PSD, hence $p_j(\theta)$ is convex $\forall \theta \in \Omega_m, \forall j=1,...,n-m$. 
   \end{proof}
\begin{remark}
Proposition~\ref{prop:equality_constraints} is one of the main contributions of this work. The explicit availability of function $p(\cdot)$ in  closed form is not guaranteed always and $p$ is often estimated iteratively in practice, which is also the rationale for including the penalizing term $\lambda_h\|h(x)\|^2$ in the definition of the loss function $\mathcal{L}(\cdot)$. In the absence of information on $p$, we can use the chain rule and implicit function theorem to calculate $\nabla_p h$ as demonstrated in~\cite{9993120},~\cite{PLD-DR-JZK:21}.
\end{remark} 

\subsection{Numerical Example}\label{section-A}
The following example illustrates how the conditions of Proposition \ref{prop:cvx-nl} and Proposition \ref{prop:equality_constraints} can be verified. We consider the problem in \eqref{eq:problem-equality-nl} where $x \in \real^2$, 
$\Lc$ is the same as $f$.
Now we demonstrate the application of Propostion~\ref{prop:equality_constraints} to define region $\Omega_n$ where $p(\cdot)$ is convex or concave.

Clearly, $\nabla_p h(x) = 2(p(x_1)-4)$, and it is evident that $\nabla_p h(x) < 0$ for $x_2 < 4$ and $\nabla_p h(x) > 0$ for $x_2 > 4$ \textit{(as $x_2=p(x_1)$)}. Using Proposition~\ref{prop:equality_constraints}, we conclude that for $x_2 \le 4$,  $p(\cdot)$ is convex and for $x_2 \ge 4$,  $p(\cdot)$ is concave. Additionally, to ensure that to $p(x_1)$ evaluates to a real number, we must constrain $-1 \leq x_1 \leq 1$. Thus, following Proposition~\ref{prop:equality_constraints}, we construct sets $\bar{\Omega}^1_n$ and $\bar{\Omega}^2_n$ as:
\begin{align}
    \bar{\Omega}^1_n = \{x = [x_1 \; \; x_2]^T | -1\le x_1\le 1, x_2 < 4\} \\
    \bar{\Omega}^2_n = \{x = [x_1 \; \; x_2]^T | -1\le x_1\le 1, x_2 > 4\} 
\end{align}
Since $\Lc(x_1, x_2) = \log(e^{x_1}+e^{x_2})$ it is evident that:
\begin{subequations}
    \begin{align}
         \nabla_{x_1} \Lc = \frac{e^{x_1}}{e^{x_1}+e^{x_2}} \\
        \nabla_{x_2} \Lc = \frac{e^{x_2}}{e^{x_1}+e^{x_2}} 
    \end{align}
\end{subequations}
Clearly, $\nabla \Lc(x_1,x_2) >0$ for all $x \in \real^2$. From Proposition~\ref{prop:cvx-nl}, case (ii), we can see that for $-1 \le x_1\le 1$ and $x_2 \le 4$, $\nabla \Lc(x_1,x_2) >0$ and $p(\cdot)$ is convex. Therefore, $l(x_1)$ is convex in this region. Formally, we define $\Omega_n \equiv \bar{\Omega}^1_n$ as:
\begin{equation}\label{eq:Omega_n}
    \Omega_n = \{x = [x_1 \; \; x_2]^T | -1\le x_1\le 1, x_2 \le 4\}
\end{equation} 
Within the chosen $\Omega_n$, Proposition-\ref{prop:cvx-nl} guarantees that $l(x_1)$ is convex. Consequently, $\Omega_m$ is automatically defined as:
\begin{equation}
    \Omega_m = \{x_1 \in \real | -1 \le x_1 \le 1\}
\end{equation}
Indeed, for values of $x_1 \in \Omega_m$ there are two cases for $p(x_1)$ using Implicit Function Theorem~\cite{Krantz2013} given as:
\begin{equation}\label{eq:p1}
    p(x_1) = 4-\sqrt{1-{x_1}^2} \quad \text{for} \; x_2 \le 4
\end{equation}
\begin{equation}\label{eq:p2}
    p(x_1) = 4+\sqrt{1-{x_1}^2} \quad \text{for} \; x_2 \ge 4
\end{equation}
As we seek to expand the set $\Omega_n$ where $l(\cdot)$ is convex, we note that due to the simple nature of the problem, it is easy to conclude that $l(\cdot)$ is concave for $x_2 \ge 4$, hence $\Omega_n$ for this problem is equivalent to the one in~\eqref{eq:Omega_n}. 

Hence, using $\Omega_n$ as defined in ~\eqref{eq:Omega_n}, $p(\cdot)$ takes the form mentioned in~\eqref{eq:p1}, and by defining $l(\cdot)$ as the one in~\eqref{eq:l_theta_expansion}, the optimization problem reduces to the following convex optimization problem which is much simpler to solve:
\begin{equation}\label{eq:constrained-eg}
    \begin{aligned}
       \min \ &\log(e^{x_1} + e^{4-\sqrt{1-x_1^2}}) \\
       \text{s.t.} \enspace & x_1 \in \Omega_m
    \end{aligned}
\end{equation}
We now state a HT based Algorithm which guarantees convergence to optimal solution provided the condition $\theta \in \Omega_m$ is satisfied.
\subsection{High Order Tuner}
Recently a high-order tuner (HT) was proposed for convex functions and shown to lead to convergence of a loss function. Since this paper builds on this HT, we briefly summarize the underlying HT algorithm presented in~\cite{Jose2021-LCSS} in the form of Algorithm~\ref{algorithm-1}.

The normalizing signal $\mathcal{N}_k$ is chosen as:
\begin{equation*}
    \mathcal{N}_k = 1 + H_k, \quad H_k = max\{\lambda: \lambda \in \sigma(\nabla^2L_k(\theta))\}
\end{equation*}
Here $\sigma(\nabla^2L_k(\theta))$ denotes the spectrum of Hessian Matrix of the loss function\cite{Jose2021-LCSS}.  
\begin{algorithm}[H]
\caption{HT Optimizer for equality-constrained nonconvex optimization}
\begin{algorithmic}[1] 
\label{algorithm-1}
\STATE \textbf{Initial conditions} $\theta_0$, $\nu_0$, gains $\gamma$, $\beta$
\FOR{$k=1$ to $N$}
\STATE Compute $\nabla l(\theta)$ and let \(\mathcal{N}_k=1+H_k\)
\STATE $\nabla \overline{q}_k(\theta_k)= \dfrac{\nabla l(\theta_k)}{\mathcal{N}_k}$
\STATE $\overline{\theta}_k=\theta_k-\gamma\beta\nabla \overline{q}_k(\theta_k)$
\STATE $\theta_{k+1}\leftarrow \overline{\theta}_k-\beta(\overline{\theta}_k-\nu_k) $
\STATE Compute $\nabla l(\theta_{k+1})$ and let
\STATE $\nabla \overline{q}_k(\theta_{k+1})= \dfrac{\nabla l(\theta_{k+1})}{\mathcal{N}_k}$
\STATE $\nu_{k+1} \leftarrow \nu_k -\gamma\nabla \overline{q}_k(\theta_{k+1})$
\ENDFOR
\end{algorithmic}
\end{algorithm}
In Section~\ref{section-A} we established sufficient conditions for the convexity of the modified loss function in Proposition~\ref{prop:cvx-nl} which are easily verifiable. We now state Theorem~\ref{thm:nl} that shows the convergence of HT in Algorithm~\ref{algorithm-1} for the constrained optimization problem in~\eqref{eq:problem-equality-nl}. 
\begin{theorem}\label{thm:nl}
If the objective function $f$ and the equality constraint $h$ in~\eqref{eq:problem-equality-nl} are convex over a set $\Omega_n$, and $\theta_0 \in \Omega_m$, 
where $\Omega_m$ is defined in~\eqref{eq:omega-m},
and if the sequence of iterates $\{\theta_k\}$ generated by Algorithm 1 satisfy $\{\theta_k\} \in \Omega_m$,
then $\underset{k\rightarrow\infty}{\lim}l(\theta_k) = l(\theta^*)$, 
where $l(\theta^*)=f([\theta^{*T} \quad p(\theta^*)^T]^T)$ is the optimal value of~\eqref{eq:problem-equality-nl} where $\gamma$ and $\beta$ are chosen as  \(0<\beta<1\), \(0<\gamma<\frac{\beta(2-\beta)}{8+\beta}\)
\end{theorem}
\begin{proof}
Readers are referred to~\cite[Theorem 2]{Jose2021-LCSS} for the proof.
\end{proof}

\subsection{Satisfaction of Box constraints}
Theorem~\ref{thm:nl} enables us to leverage Algorithm~\ref{algorithm-1}, provided that $\theta_k \in \Omega_m \forall k \in \mathbb{N}$, i.e., the parameter remains inside the set over which $l(\cdot)$ is convex. For the numerical example outlined in Section~\ref{section-A}, it is clear that we can utilize Theorem~\ref{thm:nl} when $x_1\in \Omega_1$ where $\Omega_1=[-1,1]$. It should however be noted that Theorem~\ref{thm:nl} requires that $\theta\in\Omega_m$ for the HT to lead to convergence.
In order to constrain the parameter to be within the set $\Omega_m$, we modify~\eqref{eq:unconstrained-theta} into a constrained optimization problem given by:
\begin{equation}\label{eq:constrained-theta}
    \begin{aligned}
        \min \ &l(\theta) \\
        \text{s.t.} \enspace & \theta \in \Omega_m.
    \end{aligned}
\end{equation}
It must be noted that the set $\Omega_m$ in~\eqref{eq:constrained-theta} is a subset of the feasible solution-space  of~\eqref{eq:problem-equality-nl}
If the conditions given by Proposition~\ref{prop:cvx-nl} are necessary and sufficient,~\eqref{eq:constrained-theta} has the same optimizer $\theta^*$ as that of~\eqref{eq:unconstrained-theta}.  
 While it is tempting to apply a projection procedure for ensuring $\theta \in \Omega_m$, the lack of convexity guarantees of $l(\cdot)$ for all $\theta \in \real^m$ inhibits us from proving the stability of Algorithm~\ref{algorithm-1} as we need $l(\cdot)$
to be convex for all $\overline{\theta}_k, \theta_k, \nu_k \in \real^m$. In the next section, we delineate a general procedure for ensuring that the constraint $\theta \in \Omega_m$ is always satisfied.

\section{Constrained Convex Optimization}
The starting point of this section is problem \eqref{eq:constrained-theta} where $l$ is convex.
Without loss of generality, we assume $\Omega_m$ to be a compact set in $\real^m$. For any such $\Omega_m$, it is always possible to find a bounded interval $I = I_1 \times I_2 ... \times I_m \subseteq \Omega_m$ where $I_i$ is a bounded interval in $\real$ defined as $I_i = [\theta^i_{min}, \ \theta^i_{max}]$  for all $i=1,2,...,m$. 
Using the above arguments, we reformulate the constrained optimization in~\eqref{eq:constrained-theta} as:
\begin{equation}\label{eq:problem_2}
    \begin{aligned}
        \min \ &l(\theta) \\
        \text{s.t.} \enspace & \theta \in I \\
    \end{aligned}
\end{equation}
where $\theta^*$ is the solution of problem~\eqref{eq:problem_2}. In Section~\ref{sec:nonlinear}, we outlined conditions for which~\eqref{eq:constrained-theta} is equivalent to~\eqref{eq:unconstrained-theta}, i.e., $\theta^*$ is the solution to problem~\eqref{eq:unconstrained-theta} by appropriate selection of sets $\Omega_n, \Omega_m$.

Note that~\eqref{eq:constrained-theta} and~\eqref{eq:problem_2} are equivalent if $\theta^* \in I$. The conditions under which $\theta^* \in I$ have been summarized in Proposition~\ref{prop: rolle}. 
\begin{proposition}\label{prop: rolle}
For a given compact set $I$ and convex loss function $l(\theta)$ if there exist $\theta_1, \theta_2$ such that $l(\theta_1) = l(\theta_2)$, then $\theta^* \in I$.
\end{proposition}
\begin{proof}
For a scalar case, i.e., $\Omega_m \subset  \real$, Rolle's theorem can be applied to the function $l$ being continuous and differentiable. For a subset $[\theta_{1}, \theta_{2}] \subset I$, such that $l(\theta_1) = l(\theta_2)$, by Rolle's Theorem, there exists a $\hat{\theta}\in [\theta_{1}, \theta_{2}]$ such that $\nabla l(\hat{\theta}) = 0$. Since $l$ is differentiable and convex, $\nabla l(\hat{\theta}) = 0 \iff \hat{\theta} = \theta^*$.
This can be extended to the general case where $m \ge 1, \Omega_m\subset\real^m$, using the vector-version of Rolle's theorem, cf.~\cite{doi:10.1080/00029890.1995.11990564}.
\end{proof}
With the assurance from Proposition~\ref{prop: rolle} that $\theta^* \in I$, we now revise Algoirithm~\ref{algorithm-1} in the form of Algorithm~\ref{algorithm-3} to guarantee convergence to $\theta^*$ while ensuring feasibility that $\theta \in I$. We prove that $\overline{\theta}_k, \nu_k, \theta_k \in I$ for all $k \in \mathbb{N}$ through Proposition~\ref{prop:a_k} and subsequently provide guarantees of convergence of the iterates $\{\theta_k\}_{k=1}^{\infty}$ generated by Algorithm~\ref{algorithm-3} to $\theta^*$ using Theorem~\ref{thm:projection}. 
\begin{algorithm}[H]
    \caption{HT Optimizer for localized convex optimization}
        \begin{algorithmic}[1]
        \label{algorithm-3}
            \STATE \textbf{Initial conditions} $\theta_0$, $\nu_0$, gains $\gamma$, $\beta$
            \STATE Choose $\theta_0, \nu_0 \in I$
            \FOR{$k=1$ to $N$}
            \STATE Compute $\nabla l(\theta)$ and let \(\mathcal{N}_k=1+H_k\)
            \STATE $\nabla \overline{q}_k(\theta_k)= \dfrac{\nabla l(\theta_k)}{\mathcal{N}_k}$
            \STATE $\overline{\theta}_k=\theta_k-\gamma\beta a_k\nabla \overline{q}_k(\theta_k)$
            \STATE $\theta_{k+1}= \overline{\theta}_k-\beta(\overline{\theta}_k-\nu_k) $
            \STATE Compute $\nabla l(\theta_{k+1})$ and let
            \STATE $\nabla \overline{q}_k(\theta_{k+1})= \dfrac{\nabla l(\theta_{k+1})}{\mathcal{N}_k}$
            \STATE $\nu_{k+1} \leftarrow \nu_k -\gamma b_{k+1}\nabla \overline{q}_k(\theta_{k+1})$
            \ENDFOR
        \end{algorithmic}
\end{algorithm}
\begin{proposition}\label{prop:a_k}
    Consider Algorithm~\ref{algorithm-3}, for a given $k \in \mathbb{N}$, if $\theta_k, \nu_k \in I$, there exist real numbers $a_k > 0$ and $b_{k+1} > 0$ such that $\overline{\theta}_k, \nu_{k+1} \in I$. Consequently, for $0<\beta \le 1$, and $\theta_0, \nu_0 \in I$, Algorithm~\ref{algorithm-3} guarantees $\theta_{k}, \overline{\theta}_k, \nu_k \in I$ for all values of $k \in \mathbb{N}$.  
\end{proposition}
\begin{proof}
We first provide conditions for the selection of $a_k$.

For a given $\theta_k \in \real^m$ and $i \in \mathbb{N}$, consider $\theta_k^{i} \in \real$ such that $\theta_k^{i} \in [\theta_{min}^i, \; \theta_{max}^i]$. There are two possible cases:
\begin{enumerate}
    \item $\theta_k^{i} > \theta^{*i} \iff \nabla_i l(\theta_k) > 0$
    \item $\theta_k^{i} < \theta^{*i} \iff \nabla_i l(\theta_k) < 0$
\end{enumerate}
For case (ii), using \textbf{Step-6} of Algorithm~\ref{algorithm-3} we have, 
\begin{equation}\label{eq:theta-bar}
    \overline{\theta}^{i}_{k} = \theta^{i}_{k} + a_k\frac{\gamma\beta|\nabla_i l(\theta_{k})|}{\mathcal{N}_{k}}
\end{equation}
For $a_k > 0$, $\overline{\theta}^{i}_{k} > \theta^{i}_{min}$ in~\eqref{eq:theta-bar}. We need to ensure that $\overline{\theta}^{i}_{k} < \theta^{i}_{max}$ for all $i$. Therefore, we must ensure:
\begin{align}\label{eq:case1-pre}
    \theta_{k}^{i} + a_k\frac{\gamma\beta|\nabla_{i} l(\theta_{k})|}{\mathcal{N}_{k}} \leq \theta_{max}^{i} \quad \forall i.
    \end{align}
 Inequality~\eqref{eq:case1-pre} would be true if $a_k \le \hat{a}_k$, where 
  \begin{align}\label{eq:case1}
     \hat{a}_k = \min\limits_{i \in \{1,\ldots,m\}} \frac{(\theta_{max}^{i}-\theta_{k}^{i})\mathcal{N}_{k}}{\gamma\beta|\nabla_{i} l(\theta_{k})|}.
\end{align}
Similarly, for case (i), we get the following inequality criteria for $a_k$ to ensure that $\overline{\theta}^{i}_{k} > \theta^{i}_{min}$ for all $i$:
\begin{equation}\label{eq:case2}
    a_k \leq \tilde{a}_k = \min\limits_{i \in \{1,\ldots,m\}} \frac{(\theta_{k}^{i}-\theta_{min}^{i})\mathcal{N}_{k}}{\gamma\beta|\nabla_{i} l(\theta_{k})|}
\end{equation}
Combining~\eqref{eq:case1} and~\eqref{eq:case2}, we have
\begin{align}\label{eq:a_k}
    a_k \le \bar{a}_k = \min \{\hat{a}_k,\tilde{a}_k\} \quad \forall k.
\end{align}
We now outline conditions for selection of $b_{k+1}$, which follows similar approach to selection of $a_k$, i.e., for given $\nu_k \in I$ we prescribe range of $b_{k+1}$ such that $\nu_{k+1} \in I$. 

From \textbf{Step-10} of Algorithm~\ref{algorithm-3}, it could be deduced that for all $k$, $b_{k+1}$ must satisfy
\begin{equation}\label{eq:b_kp1}
    b_{k+1}\leq \bar{b}_{k+1} = \min \{\hat{b}_{k+1} , \tilde{b}_{k+1}\}
\end{equation}
where 
\begin{align*}
   \hat{b}_{k+1} =& \min\limits_{i \in \{1,\ldots,m\}}\frac{(\theta_{max}^i-\nu_{k}^i)\mathcal{N}_{k}}{\gamma|\nabla_i l(\theta_{k+1})|} \\
   \tilde{b}_{k+1} =& \min\limits_{i \in \{1,\ldots,m\}} \frac{(\nu_{k}^i-\theta_{min}^i)\mathcal{N}_{k}}{\gamma|\nabla_i l(\theta_{k+1})|}.
\end{align*}
Note however, that~\eqref{eq:a_k},~\eqref{eq:b_kp1} can generate $a_k, b_{k+1} = 0$, which is undesirable. To compensate for that, we introduce an additional rule:
\begin{equation}\label{eq:exception}
    a_k = \begin{cases}
  -\epsilon  & \min\limits_{i \in \{1,\ldots,m\}} (\theta^i_{k}- \theta^i_{max})=0 \\
  \epsilon & \min\limits_{i \in \{1,\ldots,m\}} (\theta^i_{k}-\theta^i_{min})=0 
\end{cases} 
\end{equation}
Here $0<\epsilon<1$ is a very small real number of choice. Similar update rule can be stated for $b_{k+1}$ to avoid the case of $a_k, b_{k+1}=0$. 
 For a given $k$, by selecting $a_k$, $b_{k+1}$ such that~\eqref{eq:a_k},~\eqref{eq:b_kp1},~\eqref{eq:exception} are satisfied, we ensure $\overline{\theta}_k , \nu_k \in I$. From \textbf{Step-7} of Algorithm~\ref{algorithm-3}:
\begin{equation}\label{eq:theta_bound}
        \theta_{k+1} = (1-\beta)\overline{\theta}_k + \beta \nu_k
\end{equation}
Hence, $\theta_{k+1}$ is a convex combination of $\overline{\theta}_k$ and $\nu_k$ for a given $k$ and $0 < \beta \le 1$. Additionally, set $I$ is compact, hence, if $\overline{\theta}_k, \nu_k \in I$, then $\theta_{k+1} \in I$ for a given $k \in \mathbb{N}$. By choosing $\theta_0, \nu_0 \in I$, by induction it can be shown that Proposition~\ref{prop:a_k} can be applied iteratively to generate parameters that are bounded within the compact set I. 
\end{proof}

Proposition~\ref{prop:a_k} outlines conditions under which all parameters generated by Algorithm-\ref{algorithm-3} are bounded within set $I$, where the convexity of loss function $l(\cdot)$ is guaranteed. Theorem~\ref{thm:projection} formally establishes the convergence of Algorithm~\ref{algorithm-3} to an optimal solution $\theta^*$. 
\begin{theorem}\label{thm:projection}
    \longthmtitle{Convergence of the HT algorithm constrained to ensure convexity}
For a differentiable $\bar{L}_k$-smooth convex loss function $l(.)$, Algorithm~\ref{algorithm-3} with $0 < \beta \leq 1$, $0<\gamma<\frac{\beta(2-\beta)}{8+\beta}$ and $a_k,b_{k+1}$ satisfying
$a_k \le \min\{1,\bar{a}_k\}$, $b_{k+1} \le \min\{1,\bar{b}_{k+1}\}$ and~\eqref{eq:exception}, where
$\bar{a}_{k+1}$ and $\bar{b}_{k+1}$ are defined in~\eqref{eq:a_k} and \eqref{eq:b_kp1} ensures that $V = \frac{\lVert\nu-\theta^*\rVert^2}{\gamma}+ \frac{\lVert\nu-\theta\rVert^2}{\gamma}$ is a Lyapunov function. Consequently, the sequence of iterates $\{\theta_k\}$ generated by Algorithm~\ref{algorithm-3} satisfy $\{\theta_k\} \in \Omega_m$,
and $\underset{k\rightarrow\infty}{\lim}l(\theta_k) = l(\theta^*)$, 
where $l(\theta^*)=f([\theta^{*T} \quad p(\theta^*)^T]^T)$ is the optimal value of~\eqref{eq:constrained-theta}. 
\end{theorem}
\begin{proof}
This proof follows a similar approach to the proof of stability of High Order Tuner for convex optimization, as illustrated in ~\cite[Theorem 2]{Jose2021-LCSS}. 

Assuming that $\nu_k, \theta_k, \overline{\theta}_k \in I$, function $l(\cdot)$ is convex for all these parameters lying within the set $I$. Applying convexity and smoothness properties (ref. ~\cite[Section II]{Jose2021-LCSS}) to $l(\cdot)$, the following upper bound is obtained:
    \begin{equation*}
        l(\vartheta_k)-l(\bar{\theta}_k)=l(\vartheta_k)-l(\theta_{k+1})+l(\theta_{k+1})-l(\bar{\theta}_k)
    \end{equation*}
    \vspace{0cm}
    \begin{equation}\label{eq:two}
        \begin{split}
            &{\leq} \nabla l(\theta_{k+1})^T(\vartheta_{k}-\theta_{k+1})+\frac{\bar{L}_k}{2}\lVert\vartheta_{k}-\theta_{k+1}\rVert^2\\
            &\quad +\nabla l(\theta_{k+1})^T(\theta_{k+1}-\bar{\theta}_k)
        \end{split}
    \end{equation}
        \begin{equation}\label{eq:three}
        \begin{split}
            &\overset{\text{Alg.} \ref{algorithm-3}}{\leq} \nabla l(\theta_{k+1})^T(\vartheta_{k}-\bar{\theta}_k)+\frac{\bar{L}_k}{2}\lVert\vartheta_{k}-(1-\beta)\bar{\theta}_k-\beta\vartheta_k)\rVert^2
        \end{split}
    \end{equation}
    \begin{equation} \label{eq:UB_result1}
        \begin{split}
            &l(\vartheta_k)-l(\bar{\theta}_k)\\
            &\leq -\nabla l(\theta_{k+1})^T(\bar{\theta}_k-\vartheta_{k})+\frac{\bar{L}_k}{2}(1-\beta)^2\lVert\bar{\theta}_k-\vartheta_k\rVert^2.
        \end{split}
    \end{equation}
    Similarly, we obtain:
    \begin{equation} \label{eq:UB_result2}
        \begin{split}
            &l(\bar{\theta}_k)-l(\vartheta_k)\\
            &\quad\leq \nabla l(\theta_{k})^T(\bar{\theta}_k-\vartheta_{k})+\frac{a_k^2\bar{L}_k\gamma^2\beta^2}{2\N_k^2}\lVert \nabla l(\theta_k)\rVert^2.
        \end{split}
    \end{equation}
    Using \eqref{eq:UB_result1} and \eqref{eq:UB_result2} we obtain:
    \begin{equation}\label{eq:mainresult_1}
        \begin{split}
            &\nabla l(\theta_{k+1})^T(\bar{\theta}_k-\vartheta_{k})\\
            &-\frac{\bar{L}_k}{2}(1-\beta)^2\lVert\bar{\theta}_k-\vartheta_k\rVert^2\\
            &-\frac{a_k^2\bar{L}_k\gamma^2\beta^2}{2\N_k^2}\lVert \nabla l(\theta_k)\rVert^2\leq \nabla l(\theta_{k})^T(\bar{\theta}_k-\vartheta_{k})
        \end{split}
    \end{equation}
    
    Using Algorithm \ref{algorithm-3}, \cite[Theorem 1]{Jose2021-LCSS} and \eqref{eq:mainresult_1}, setting $\gamma<\frac{\beta(2-\beta)}{8+\beta}$, $0< a_k, b_{k+1} \le 1$ and defining $\Delta V_k:=V_{k+1}-V_k$, it can be shown that 
    \begin{align}\label{eq:key-step}
    \Delta V_k \leq \frac{1}{\mathcal{N}_k}\Bigg\{
    -2b_{k+1}(l(\theta_{k+1})-l(\theta^*)) \nonumber \\
    -\Bigg(\frac{b_{k+1}}{2\bar{L}_k}-\frac{2\gamma b_{k+1}^2}{\mathcal{N}_k}\Bigg)\lVert \nabla l(\theta_{k+1})\rVert^2 \nonumber \\
    -\Bigg(1-\frac{\bar{L}_k\gamma\beta a_k}{\mathcal{N}_k}\Bigg)\frac{\gamma\beta^2a^2_k}{\mathcal{N}_k}\lVert \nabla l(\theta_k)\rVert^2 \nonumber \\
    -[\beta-a_k\beta(1-\beta)^2]\bar{L}_k\lVert \overline{\theta}_k-\nu_k \rVert^2 \nonumber \\
    -\Bigg[\frac{\sqrt{b_{k+1}}}{\sqrt{2\bar{L}_k}}\lVert \nabla l(\theta_{k+1})\rVert -2\sqrt{2\bar{L}_k}\lVert\overline{\theta}_k-\nu_k\rVert\Bigg]^2 \nonumber \\
    -4(\sqrt{b_{k+1}}-b_{k+1})\lVert \overline{\theta}_k-\nu_k \rVert \lVert \nabla l(\theta_{k+1})\rVert  \nonumber \\
    -(8+\beta)\lVert\overline{\theta}_k-\nu_k\rVert^2 \Bigg\} \leq 0
\end{align}
From~\eqref{eq:key-step}, it can be seen that:
\begin{equation}
     \Delta V_k \leq \frac{b_{k+1}}{\mathcal{N}_k}\{
    -2(l(\theta_{k+1})-l(\theta^*)\} \leq 0
\end{equation}
Collecting $\Delta V_k$ terms from $t_0$ to $T$, and letting $T\rightarrow \infty$, it can be seen that  $l(\theta_{k+1})-l(\theta^*)
\in\ell_1\cap\ell_\infty$ and therefore $\lim_{k\rightarrow\infty}l(\theta_{k+1})
-l(\theta^*)=0.$
\end{proof}
\section{Numerical Study}
\subsection{An academic example}
We consider the same example as before which led to the following constrained convex optimization problem:
\begin{equation}\label{eq:projection-example-1}
    \begin{aligned}
      \min \ &\log(e^{x_1} + e^{4-\sqrt{1-x_1^2}}) \\
      \text{s.t.} \enspace &-1 \leq x_1 \leq 1
    \end{aligned}
\end{equation}
The box constraint in~\eqref{eq:projection-example-1} is  imposed to ensure that the objective function is always real-valued. While we can choose certain step-sizes that ensure that $-1 \leq x_1 \leq 1$ for solving~\eqref{eq:projection-example-1}, there are no guarantees that such a step-size exists to ensure this constraint. Algorithm~\ref{algorithm-3} solves this problem with suitable choices of $\gamma$, $\beta$, $a_k$ and $b_{k+1}$ as specified in Theorem~\ref{thm:projection}.  Figure~\ref{fig:Numerical-example-1} shows the convergence of parameter $\theta$ (equivalent to $x_1$ in~\eqref{eq:projection-example-1}) using Algorithm~\ref{algorithm-3} for a chosen value of $\gamma, \beta$. It should be noted in Figure~\ref{fig:Numerical-example-1} that Algorithm~\ref{algorithm-1} fails to converge. This illustrates the value of Theorem III.3, another contribution of this paper.
\begin{figure}
    \includegraphics[width=0.4\textwidth]{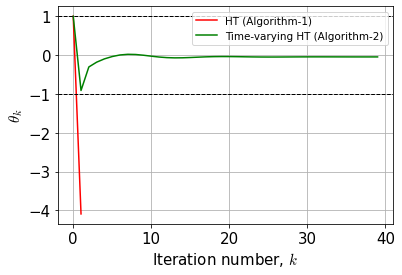}
    \caption{Convergence of $\theta$ using Algorithm~\ref{algorithm-3} for the problem in~\eqref{eq:projection-example-1}; Algorithm~\ref{algorithm-1} fails to converge}
    \label{fig:Numerical-example-1}
\end{figure}
\subsection{Provably hard problem of Nesterov}
We consider a provably hard problem which corresponds to a strongly convex function (see~\cite{Jose2021-LCSS} for details)
\begin{equation}\label{eq:simulation_loss_mu}
    l(\theta)=\log{(c_ke^{d_k\theta}+c_ke^{-d_k\theta})}+\frac{\mu}{2}\lVert\theta-\theta_0\rVert^2.
\end{equation}
Here $c_k$ and $d_k$ are positive scalars chosen as $c_k = \frac{1}{2}$ and $d_k = 1$. This function has a unique minimum at $\theta^*=0$. 
In all cases, $\mu=10^{-4}$, the intial value was chosen to be $\theta_0=2$ and the constraints are chosen as $\theta_{min}=-1$ and $\theta_{max}=2$. It is clear from Figure \ref{fig:simulation-theta} that with Algorithm~\ref{algorithm-3}, the parameters converge to the optimal value while being within $[\theta_{min}, \theta_{max}]$. The speed of convergence is faster than that of Algorithm~\ref{algorithm-1} and Algorithm~\ref{algorithm-3} exhibits lesser oscillations, which is an attractive property. 
\begin{figure}
    \includegraphics[width=0.4\textwidth]{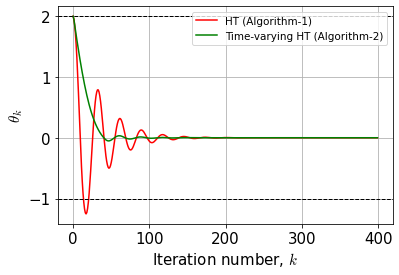}
    \caption{Constraint satisfaction by Algorithm-2, Algorithm-1 generates iterates that violate the box constraints [-1,2]. Algorithm-2 displays smaller amplitude oscillations compared to Algorithm-1}
    \label{fig:simulation-theta}
\end{figure}

\section{Conclusion and Future Work}
In this work, we extend the previously conducted study on using High Order Tuners to solve constrained convex optimization. We propose a new HT that can accommodate the constraints and the non-convexities with guaranteed convergence. We provide academic examples and numerical simulations to validate the theorems presented in the paper pertaining to the accelerated convergence and feasibility guarantees for High Order Tuner. This work establishes a framework to explore potential applications of constrained optimization that would benefit from faster convergence, such as neural network training which can be reformulated in some cases as convex optimization problems~\cite{pmlr-v119-pilanci20a}, and solving Optimal Power Flow (OPF) problems using neural networks as in~\cite{Pan:21}, \cite{inproceedings}. 
\bibliographystyle{IEEEtran}
\bibliography{alias.bib,main.bib}
\end{document}